\documentclass[a4paper,12pt]{amsart}
\usepackage[francais]{babel}
\usepackage[latin1]{inputenc}
\usepackage{pdfsync}
\usepackage[T1]{fontenc}
\usepackage{amsmath}
\usepackage{amssymb}
\usepackage{amsxtra}
\usepackage{amscd}
\usepackage{amsthm}
\usepackage{amsfonts}
\usepackage{eucal}
\usepackage[all]{xy}
\usepackage{graphicx}
\usepackage{comment}
\usepackage{epsfig}
\usepackage{psfrag}
\usepackage{mathrsfs}
\usepackage{amscd}
\usepackage{rotating}
\usepackage{lscape}
\usepackage{amsbsy}
\usepackage{verbatim}
\usepackage{moreverb}
\usepackage{url}
\usepackage[pdftex]{color}

\newtheorem{thm}{Théorème}[section]
\newtheorem{prop}[thm]{Proposition}
\newtheorem{lem}[thm]{Lemme}
\newtheorem{sous-lem}[thm]{Sous-lemme}
\newtheorem{cor}[thm]{Corollaire}

 \theoremstyle{definition}

\newtheorem{rem}[thm]{Remarque}

\numberwithin{equation}{section}

\newcommand{\nc}{\newcommand}
\nc{\renc}{\renewcommand}

\nc\restr[2]{{
  \left.\kern-\nulldelimiterspace 
  #1 
  \vphantom{\big|} 
  \right|_{#2} 
  }}

\renc{\sec}{\section}
\nc{\ssec}{\subsection}
\nc{\sssec}{\subsubsection}

\nc{\on}{\operatorname}

\nc\wt{\widetilde}
\nc\wh{\widehat}
\nc\ol{\ov}
\nc{\oc}[1]{{\overset{\circ}{#1}}}
\nc{\ov}[1]{{\overline{#1}}}
\nc{\isor}[1]{{\xrightarrow[\raisebox{0.25 em}{\smash{\ensuremath{\sim}}}]{#1}}}
\nc{\isol}[1]{{\xleftarrow[\raisebox{0.25 em}{\smash{\ensuremath{\sim}}}]{#1}}}
\nc{\modmod}{/ \! \! /}

\nc{\mc}{\mathcal}
\nc{\mf}{\mathfrak}
\nc{\mr}{\mathrm}
\nc{\mb}{\mathbb}
\nc{\mbf}{\mathbf}

\nc{\R}{{\mathbb R}}
\nc{\Z}{{\mathbb Z}}
\nc{\N}{{\mathbb N}}
\nc{\C}{{\mathbb C}}
\nc{\Q}{{\mathbb Q}}

\nc{\Fq}{{\mathbb F}_q}
\nc{\Fl}{{\mathbb F}_\ell}
\nc{\Fqbar}{\ol{{\mathbb F}_q}}
\nc{\Flbar}{\ol{{\mathbb F}_\ell}}
\nc{\Zl}{{\mathbb Z}_\ell}
\nc{\Zlbar}{\ol{{\mathbb Z}_\ell}}
\nc{\Ql}{{\mathbb Q}_\ell}
\nc{\Qlbar}{\ol{{\mathbb Q}_\ell}}
\nc{\hl}{\overset{\leftarrow}h{}}
\nc{\hr}{\overset{\rightarrow}h{}}
\nc{\Gr}{{\on{Gr}}}
\nc{\Hecke}{\on{Hecke}}
 \nc{\Hom}{\on{Hom}}
 \nc{\Coker}{\on{Coker}}
 \nc{\Ker}{\on{Ker}}
 \nc{\Lie}{\on{Lie}}
\nc{\Loc}{\on{Loc}}
\nc{\Pic}{\on{Pic}}
\nc{\Bun}{\on{Bun}}
\nc{\IC}{\on{IC}}
\nc{\Aut}{\on{Aut}}
\nc{\Perv}{\on{Perv}}
\nc{\pos}{{\on{pos}}}
\nc{\Sym}{\on{Sym}}

\nc{\ta} {{}^\tau}
\nc {\tu}[1]{{}^{\tau^{#1}}\!}
\nc{\tav} {{}^\sigma}
\nc {\tuv}[1]{{}^{\sigma^{#1}}\!}

\nc{\Chr}{\mr{Cht}\mc R}
\nc{\Id}{\on{Id}}
\nc{\Fil}{\on{Fil}}
\nc{\pr}{\on{pr}}
\nc{\Res}{\on{Res}}
\nc{\cusp}{\on{cusp}}
\nc{\Frob}{\on{Frob}}
\nc{\diag}{\Delta}
\nc{\gr}{\on{gr}}
\nc{\Inj}{\on{Inj}}
\nc{\Bl}{\on{Bl}}
\nc{\dem}{\noindent {\bf Démonstration. }}
\nc{\cqfd}{{\ }\hfill $\square$ \vskip 1mm}
\nc{\s}[1]{\langle #1 \rangle}
\nc{\Cht}{\on{Cht}}
\nc{\isom}{\overset {\thicksim}{\to}}
\nc{\sm}{\smallsetminus}

\nc\Spf{\mathop{\mathrm{Spf }}}

\title{Décomposition au-dessus des paramètres de  Langlands elliptiques}
\author{Vincent Lafforgue et Xinwen Zhu}
\address{Vincent Lafforgue: CNRS et Institut Fourier, UMR 5582, Université Grenoble Alpes, 
 100 rue des Maths, 38610 Gières, France.}
\address{Xinwen Zhu:  California Institute of Technology, Pasadena, CA 91125}
\thanks{
X.Z. a été soutenu par les NSF grants DMS-1602092 et DMS-1902239.}

\begin{document}

\maketitle

Dans ce texte on étend la remarque  8.5 de \cite{ICM-text} au cas non déployé, et on justifie le fait que l'heuristique   (8.3)  de \cite{ICM-text} est vraie au-dessus des paramètres de   Langlands elliptiques $\sigma$ (comme cela est affirmé dans la remarque 8.5 de \cite{ICM-text}).  
Ce résultat avait été proposé par  le second auteur dans  \cite{zhu-notes}. 

Les paragraphes~\ref{para1} et~\ref{para2} énoncent les résultats, le paragraphe~\ref{para3} explique le lien avec les conjectures d'Arthur et Kottwitz, puis les paragraphes suivants donnent les démonstrations. 

Nous remercions Cong Xue pour des commentaires sur cet article. 

\section{Introduction}\label{para1}

Comme dans le paragraphe 12 de \cite{coh}, on fixe 
une courbe projective lisse géométriquement irréductible $X$ sur $\Fq$, de corps des fonctions $F$, et 
un groupe   réductif lisse et géométriquement connexe  $G$ sur  $F$. 
On note $U_0$ l'ouvert maximal de $X$ tel que $G$ se prolonge en un   schéma  en groupes lisse et réductif  sur  $U_0$. 
On étend $G$ en un schéma en groupes lisse sur  $X$,  réductif sur $U_0$, de type parahorique   en les  points de $X\sm U_0$, et dont toutes les fibres  sont géométriquement connexes. 
On fixe aussi une extension finie $E$ de $\Ql$ contenant une racine carrée de $q$. 

Soit $N\subset X$ un sous-schéma fini.   On note $\wh N=|N|\cup (X\sm U_0)$. Soit  $\Bun_{G,N}$ le champ classifiant les  $G$-fibrés principaux sur $X$ avec structure de niveau $N$. On choisit un réseau  $\Xi\subset Z(F)\backslash Z(\mathbb A)$. 
Soit $C_{\cusp}(\Bun_{G,N}(\mathbb F_q)/\Xi,E)$ l'espace des formes automorphes cuspidales  $\Xi$-invariantes à valeurs dans $E$. On note  $\mathbb T_N=C_c(K_N\backslash G(\mathbb A)/K_N, E)$ l'algèbre de  Hecke en niveau $N$, qui agit sur cet espace.

On note   $\wt F$  l'extension finie galoisienne de $F$  telle  que  $\on{Gal}(\wt F/F)$ soit l'image  de $\on{Gal}(\ov F/F)$ dans le groupe des  automorphismes de la donnée radicielle  de $G$. Le  $L$-groupe ${}^{L }G$ est le  produit   semi-direct
   $\wh G\rtimes \on{Gal}(\wt F/F)$    pour l'action de $\on{Gal}(\wt F/F)$ sur 
 $\wh G$ qui préserve un épinglage. 
 
On note  $\Gamma=\on{Gal}(\ov F/F)$ le groupe de  Galois de $F$. 

Soit  $\on{FinS}_*$ la  catégorie des ensembles finis pointés. 
On considère deux catégories additives cofibrées sur  $\on{FinS}_*$. 
La première, notée   $\on{Rep}(\hat{G}\times {}^L{G}^\bullet)$ 
associe à tout ensemble fini pointé $\{0\}\cup I$  la  catégorie $\on{Rep}(\hat{G}\times ({}^L{G})^I)$ des représentations algébriques $E$-linéaires de dimension finie 
de  $\hat{G}\times ({}^{L}G)^I$, et le foncteur d'image inverse par une application pointée  $\xi: \{0\}\cup I\to \{0\}\cup J$ est donné par la restriction naturelle suivant le morphisme de groupes algébriques  $\hat{G}\times ({}^L{G})^J\to \hat{G}\times ({}^L{G})^I$.
La seconde, notée  $\on{Rep}_{\mathbb T_N}(\Gamma^\bullet)$, associe à  $\{0\}\cup I$ la  catégorie  $\on{Rep}_{\mathbb T_N}(\Gamma^I)$ 
des représentations   continues de  $\Gamma^I$ sur des $\mathbb T_N$-modules $\Z$-gradués, et le foncteur  d'image inverse par une application pointée  $\xi: \{0\}\cup I\to \{0\}\cup J$ est donné par la restriction naturelle suivant le morphisme de groupes  topologiques  $\Gamma^J=\{1\}\times\Gamma^J\to\{1\}\times\Gamma^I=\Gamma^I$.
 
 On rappelle le théorème suivant. 
 
\begin{thm} (démontré pour $G$ déployé, conditionnel en général en ce qui concerne les finitudes  \cite{coh,cong, cong-finite}). 
On possède deux foncteurs $E$-linéaires  $H$ et $H^{\cusp}$ de  $\on{Rep}(\hat{G}\times {}^L{G}^\bullet)$ vers  $\on{Rep}_{\mathbb T_N}(\Gamma^\bullet)$ cofibrés sur  $\on{FinS}_*$.

Pour tout ensemble fini pointé  $\{0\}\cup I$, et toute représentation  $W$ de  $\hat{G}\times {}^{L}G^I$, on note $H_{\{0\}\cup I,W}$ le $\Gamma^I\times \mathbb T_N$-module correspondant,  et  $H^{\cusp}_{\{0\}\cup I,W}$ son sous-module ``cuspidal''. 

  En particulier on a  $H_{\{0\},\mathbf{1}}=C_{c}(\Bun_{G,N}(\mathbb F_q)/\Xi,E)$, muni  de la  $\mathbb T_N$-action naturelle, et $H^{\cusp}_{\{0\},\mathbf{1}}=C_{c}^{\cusp}(\Bun_{G,N}(\mathbb F_q)/\Xi,E)$.
  
  On a les propriétés de finitude suivantes (montrées dans \cite{cong, cong-finite} pour $G$ déployé,  conjecturées en général):  pour $ I,W$ comme ci-dessus, 
 $H^{\cusp}_{\{0\}\cup I,W}$ est un $E$-espace vectoriel de dimension finie, et, pour toute place $v$ dans $X\sm N$,   $H_{\{0\}\cup I,W}$    est de type fini sur l'agèbre de Hecke sphérique $\mathbb T_{v}$ en $v$. 
   \end{thm}
On expliquera dans le paragraphe \ref{rappels-coho} la construction des foncteurs  $H$ et $H^{\cusp}$, d'après \cite{coh}. En gros  $H_{\{0\}\cup I,W}$ est  la cohomologie à support compact d'un champ de chtoucas associé à  $W$, et $H^{\cusp}_{\{0\}\cup I,W}$ est sa partie ``cuspidale'' (au sens de \cite{cong} dans le cas déployé,  et sa partie Hecke-finie en général). Le théorème signifie plus concrètement que $H_{\{0\}\cup I,W}$ et $H^{\cusp}_{\{0\}\cup I,W}$ sont fonctoriels en $W$ et compatibles avec les isomorphismes de coalescence (nous nous référons à \cite[Proposition 9.7]{coh} pour plus de détails).
La finitude de la dimension des espaces vectoriels $H^{\cusp}_{\{0\}\cup I,W}$ est montrée dans \cite{cong} lorsque $G$ est déployé. On s'attend à ce que la preuve s'étende pour tout $G$ mais cela n'est pas encore écrit. On l'admet dans cet article.

Comme dans  \cite{coh}, on déduit du foncteur  $H^{\cusp}$ l'algèbre  d'excursion  $\mathcal B^{\cusp}$ qui agit sur chaque  $H^{\cusp}_{I, W}$, en commutant avec les actions de  $\Gamma^I\times \mathbb T_N$. De plus les algèbres  $\mathcal B^{\cusp}$ et  $\mathbb T_N$ 
contiennent les opérateurs de Hecke aux places non-ramifiées et ceux-ci agissent de la même fa\c con par les deux algèbres (grâce à  \cite{genestier-lafforgue} on a aussi une compatibilité entre les actions de $\mathcal B^{\cusp}$ et  $\mathbb T_N$ en les places ramifiées).  
 L'action de $\mathcal B^{\cusp}$  sur l'espace vectoriel gradué $H^{\cusp}_{ I, W}$ de dimension finie 
fournit une décomposition de cet espace en sous-espaces  \emph{généralisés}, 
\[
H^{\cusp}_{ I, W}\otimes_E\Qlbar=\bigoplus (H^{\cusp}_{I, W})_\sigma,
\]
où la somme directe est indexée par des caractères  $\nu: \mathcal B^{\cusp}\to\Qlbar$. Cette décomposition est respectée par l'action de 
$\Gamma^I\times \mathbb T_N^{\Qlbar}$, où $\mathbb T_N^{\Qlbar}:=\mathbb T_N\otimes_E\Qlbar$. 
D'après  \cite{coh}, tout caractère $\nu $ détermine un morphisme  semi-simple   $\sigma:\Gamma\to {}^L{G}(\Qlbar)$ à conjugaison près par  $\hat{G}$,  satisfaisant  les  conditions suivantes: 
\begin{itemize}
\item [] (C1) $\sigma$ prend ses valeurs  dans  ${}^{L}G(E')$, où $E'$ est une extension finie  de  $E$, et il  est continu, 
\item [] (C2) l'adhérence de Zariski de son image est réductive, 
\item [] (C3)  $\sigma$  se factorise à travers 
$ \pi_{1}(U, \ov\eta)$ pour un certain ouvert $U\subset U_0$, 
\item [] (C4) on a  la commutativité du   diagramme 
 \begin{gather}\label{diag-sigma}
 \xymatrix{
\on{Gal}(\ov F/F) \ar[rr] ^{\sigma}
\ar[dr] 
&& {}^{L} G(\Qlbar) \ar[dl] 
 \\
& \on{Gal}(\wt F/F) }\end{gather}
\item [] (C5) la composition de $\sigma$ avec 
\begin{gather}\label{compo-ab} {}^{L}  G(\overline{{\mathbb Q}_\ell})\to
(\wh G^{ab})(\overline{{\mathbb Q}_\ell})\rtimes \on{Gal}(\wt F/F)
\end{gather} a une image finie. 
\end{itemize}

Réciproquement $\sigma$ détermine $\nu$, c'est pourquoi on écrit $(H^{\cusp}_{I, W})_\sigma$. 

Pour tout $\sigma$ comme ci-dessus, on note  $S_\sigma$ le centralisateur de son image dans  $\hat{G}$.
Un tel $\sigma$ est dit elliptique si  $\mathfrak S_\sigma:=S_\sigma/(Z_{\hat{G}})^{\on{Gal}(\wt F/F)}$ est fini. 

Pour toute  représentation $W$ de $({}^L{G})^I$, on note $W_{\sigma^I}$ la représentation composée $\Gamma^{I}\stackrel{\sigma^I}{\to} ({}^L{G})^I\to GL(W)$.
Voici le résultat principal de cet article. 

  \begin{prop}\label{prop-ppale} (démontrée par $G$ déployé, conditionnelle en général à l'extension de \cite{cong} au cas non déployé) 
  Pour tout $\sigma$ elliptique, il existe un $\Qlbar$-espace vectoriel de dimension finie $\mathfrak A_\sigma$, muni d'une action de 
  $\mathbb T_N^{\Qlbar}\times\mathfrak S_\sigma$, tel que pour tout  $W$,  on a l'égalité 
\begin{gather}\label{eq-prop-ppale}(H^{\cusp}_{I,W})_{\sigma}= \big( \mathfrak A_{\sigma}\otimes_{\overline{{\mathbb Q}_\ell}} W_{\sigma^{I}}\big)^{S_{\sigma}},\end{gather} 
 équivariante par l'action de  $\mathbb T_N^{\Qlbar}\times \Gamma^{I}$, 
 fonctorielle en $W$, compatible avec les isomorphismes de coalescence, et  compatible avec le changement de  structure de niveau. 
Autrement dit l'heuristique  (8.3) de  \cite{ICM-text} est vraie au-dessus de $\sigma$.  
  \end{prop}

La démonstration sera donnée à partir du paragraphe \ref{non-deform}.

\begin{rem} Le membre de droite de \eqref{eq-prop-ppale} est nul lorsque 
$W^{ Z_{\hat{G}}^{\Gamma_F}}=0$. Dans ce cas la nullité du membre de gauche est évidente car le champ de chtoucas correspondant est vide. 
Le cas intéressant de \eqref{eq-prop-ppale} est donc celui où $W$ est une représentation de $ {}^L{G}^I/ Z_{\hat{G}}^{\Gamma_F}$. Dans ce cas on peut remplacer le membre de droite de  \eqref{eq-prop-ppale} par 
$\big( \mathfrak A_{\sigma}\otimes_{\overline{{\mathbb Q}_\ell}} W_{\sigma^{I}}\big)^{\mathfrak  S_{\sigma}}$. 
\end{rem}

 \begin{rem} \label{rem-propre-naif} 
 La proposition  implique que pour $\sigma$ elliptique  $(H^{\cusp}_{I,W})_{\sigma}$ est en fait le sous-espace propre au sens strict (alors qu'il est défini comme   sous-espace propre au  sens généralisé) de  $H^{\cusp}_{I,W}\otimes_{E}\overline{{\mathbb Q}_\ell} 
$
associé à  $\sigma$ pour l'action des opérateurs d'excursion, et de plus $(H^{\cusp}_{I,W})_\sigma$ est semi-simple comme représentation de  $\Gamma^I$. 
\end{rem}

 \begin{rem} En  général à cause de la  déformation possible de certains   $\sigma$ non elliptiques 
 il pourrait y avoir des nilpotents,
 et par exemple on ne sait pas montrer que  l'algèbre engendrée par les opérateurs d'excursion agissant sur $H^{\cusp}_{I,W}$ (ou même sur $H^{\cusp}_{\emptyset,\mbf 1}$) est réduite, donc on ne sait pas montrer 
 la première partie de l'énoncé de la remarque précédente pour tous les $\sigma$ et à fortiori on ne sait pas montrer l'heuristique (8.3) de \cite{ICM-text}. 
    \end{rem}

\section{Le cas de $GL_{r}$}
\label{para2}

On prend  $G=GL_{r}$ et on applique la proposition \ref{prop-ppale}. 
Alors une représentation semi-simple $\sigma$ est elliptique si et seulement si
elle est irréductible, et dans ce cas  $\mathfrak S_\sigma$ est trivial. 
Il en résulte que 
\[\mathfrak A_\sigma=\mathfrak A_\sigma^{\mathfrak S_\sigma}=(H^{\cusp}_{\emptyset,\mbf 1})_\sigma.\]
Par le théorème de multiplicité un fort pour l'espace des formes automorphes cuspidales pour  $GL_{r}$, on voit que  $\mathfrak A_\sigma=\pi^{K_N}$, où  $\pi$ est l'unique représentation cuspidale qui correspond à $\sigma$. On a donc la conséquence suivante de la proposition~\ref{prop-ppale}. 
\begin{cor}
Pour  $G=GL_{r}$, si $\sigma$ est  irréducible, alors 
\[(H^{\cusp}_{I,W})_\sigma= \pi^{K_N}\otimes (W^{\mb G_m})_{\sigma^I}.\]
\end{cor}
\begin{rem}
Dans le cas des chtoucas de Drinfeld ($G=GL_{r}$, $I=\{1,2\}$, $W=\on{St}\otimes \on{St}^{*}$) l'énoncé obtenu dans le corollaire ci-dessus est exactement celui montré par Laurent Lafforgue dans \cite{laurent-inventiones}  (où les multiplicités sont déterminées aux $r$-négligeables près). Pour  $GL_{2}$, 
cela est dû à   Drinfeld  \cite{drinfeld-Petersson}.  
\end{rem}

\begin{rem}
On conjecture que pour  $GL_{r}$ tout $\sigma$ apparaissant dans la décomposition de  $H^{\cusp}_{I,W}$ est irréductible  (cela est connu dans le cas des formes automorphes, c'est-à-dire pour $W=\mbf 1$, mais pas en général). 
Sous cette hypothèse supplémentaire on obtient  la conjecture 2.35 de Varshavsky \cite{var}.  
\end{rem}

\section{Relation avec les formules de multiplicités d'Arthur-Kottwitz}
\label{para3}

On revient au cas d'un groupe $G$ réductif arbitraire. 

La formule
\[(H^{\cusp}_{I,W})_\sigma=\big( \mathfrak A_{\sigma}\otimes_{\overline{{\mathbb Q}_\ell}} W_{\sigma^{I}}\big)^{S_{\sigma}} \]
que nous montrons dans cet article pour $\sigma$ elliptique
est  compatible avec la conjecture de 
 Kottwitz sur la partie cuspidale tempérée du spectre  automorphe (voir  \cite{cusp-temper}) et de la cohomologie des variétés de Shimura (see \cite{KoAnnArbor}). Ici bien sûr on considère  $H^{\cusp}_{I,W}$ 
 comme un analogue généralisé, dans le cas des corps de fonctions,  de la cohomologie des  variétés de Shimura.  
  On rappelle que la conjecture de 
 Kottwitz sur la partie cuspidale tempérée du spectre  automorphe 
 est un cas particulier des formules de multiplicités d'Arthur. 
On garde l'hypothèse que  $\sigma$ est elliptique.

On rappelle que  $\mathfrak A_\sigma$  est muni des actions (commutant entre elles) de l'algèbre de Hecke $\mathbb T_N^{\Qlbar}:=C_c(K_N\backslash G(\mathbb A)/K_N, \overline{{\mathbb Q}_\ell})$ et du groupe fini 
 $\mathfrak S_\sigma: = S_{\sigma}/(Z_{\wh  G})^{ \on{Gal}(\wt F/F)}$. 

On écrit 
\[\mathfrak A^{ss}_\sigma=\bigoplus \pi^{K_N}\otimes \rho_\pi,\] 
où  $\mathfrak A^{ss}_\sigma$ est la  semi-simplification de  $\mathfrak A_\sigma$ comme  $\mathbb T_N^{\Qlbar}$-module, où la somme directe est prise sur les  $\mathbb T_N^{\Qlbar}$-modules  irréductibles $\pi^{K_N}$, et où  $\rho_\pi$ est un espace de multiplicités, qui est en fait une représentation de dimension finie de   $\mathfrak S_\sigma$. On conjecture que  $\mathfrak A_\sigma$  est  semi-simple comme $\mathbb T_N^{\Qlbar}$-module. En fait, on conjecture que les $H^{\cusp}_{I,W}$  sont semi-simples comme $\mathbb T_N$-modules, mais on ne sait le montrer que pour $W=\bf 1$ (cas des formes automorphes). 

On remarque que si  un  $\mathbb T_N^{\Qlbar}$-module irréductible $\pi^{K_N}$ apparaît dans  $(H^{\cusp}_{I,W})_\sigma$, alors   $\sigma$ est elliptique  si et seulement si 
$\pi$ est tempéré aux places non-ramifiées et
la  function $L(s,\pi, \on{Ad}^0)$ partielle (restreinte aux places où  $\pi$ est non-ramifié)  
admet un prolongement méromorphe au plan complexe tout entier et n'a pas de pôle en  $s=1$, où  $\on{Ad}^0$ est la représentation  adjointe de  ${}^{L}G$ sur l'algèbre de Lie adjointe  $\hat{\mathfrak g}_{ad}$. 
En effet on a $$\on{Lie} (\mf S_{\sigma})=\on{Lie} (S_{\sigma})/\on{Lie} ((Z_{\wh  G})^{ \on{Gal}(\wt F/F)})=(\on{Ad})^{\sigma}/\on{Lie} ((Z_{\wh  G})^{ \on{Gal}(\wt F/F)})=(\on{Ad}^{0})^{\sigma}$$ où $(\on{Ad})^{\sigma}$ et $(\on{Ad}^{0})^{\sigma}$ désignent les invariants par $\sigma$. 
Par conséquent si 
$\pi^{K_N}$ apparaît dans un  $(H^{\cusp}_{I,W})_\sigma$ avec $\sigma$ elliptique,
il apparaît seulement dans des  $(H^{\cusp}_{I,W})_{\sigma'}$ avec  $\sigma'$ elliptique.

Suivant la littérature sur les formules de multiplicités, on note 
$\langle \pi, \cdot\rangle$ la fonction 
 trace $\on{tr}\rho_\pi$ sur  $\mathfrak S_\sigma$.
Il résulte de tout cela que la multiplicité de $\pi^K$ dans  $(H^{\cusp}_{\emptyset,\mbf 1})_\sigma$ est égale à 
\[m(\pi,\sigma)= \frac{1}{|\mathfrak S_\sigma|}\sum_{x\in \mathfrak S_\sigma} \langle\pi, x\rangle.\]
Il en résulte que la multiplicité totale de $\pi^{K_N}$ dans l'espace des formes automorphes cuspidales est donnée par \[m(\pi)=\sum_\sigma m(\pi,\sigma).\]
Ceci est en accord avec la formule de multiplicités de  Arthur et Kottwitz.

Soit $W$ une représentation arbitraire  de $({}^{L}  G)^I/(Z_{\wh  G})^{ \on{Gal}(\wt F/F)}$. Supposons d'abord que $\sigma$ soit un paramètre elliptique stable, ce qui signifie que $\mathfrak S_\sigma$ est trivial. Alors exactement comme dans le  corollaire 2.1, on a la conséquence suivante de la proposition~\ref{prop-ppale}. 
\begin{cor}Pour un paramètre stable $\sigma$,
\[
(H^{\cusp}_{I,W})_\sigma= (H^{\cusp}_{\emptyset,\mathbf{1}})_\sigma \otimes W_{\sigma^I}.
\]
\end{cor}
Ceci a été prouvé par Kazhdan et Varshavsky dans un preprint non publié \cite{KV}, en supposant qu'une composante locale des représentations automorphes dans $(H^{\cusp}_{\emptyset,\mathbf{1}})_\sigma$  est supercuspidale.

Pour un paramètre non nécessairement stable, la proposition~\ref{prop-ppale} implique 
\[
(H^{\cusp}_{I,W})^{ss}_\sigma = \bigoplus_{\pi} \pi^{K_N}\otimes \Hom_{\mathfrak S_\sigma}(\rho_\pi^\vee, W_{\sigma^I}),
\]
comme semi-simplifiés de   $\mathbb T_N^{\Qlbar}\times (\Gamma_F)^I$-modules (mais on rappelle que $(H^{\cusp}_{I,W})_\sigma$ est déjà semi-simple en tant que  $(\Gamma_F)^I$-module). 
En particulier la multiplicité de  $\pi^{K_{N}}$ dans $(H^{\cusp}_{I,W})^{ss}_\sigma$ est 
\[m_{I,W}(\pi,\sigma)=\frac{1}{|\mathfrak S_\sigma|}\sum_{x\in\mathfrak S_\sigma} \langle \pi,x\rangle\on{tr}(x\mid W),\]
ce qui est similaire à la formule de multiplicités de  Kottwitz pour la cohomologie des variétés de  Shimura,  conjecturée dans  \cite{KoAnnArbor}.

\bigskip

On espère la description locale suivante des espaces de multiplicités $\mathfrak A_\sigma$.

Si  $v$ est une place de $F$, soit $\sigma_v$ la restriction de $\sigma$ au groupe de décomposition  $\on{Gal}(\ov F_v/F_v)$, et soit  $S_{\sigma_v}$ le centralisateur dans $\wh G$ de son image dans  ${}^{L}  G$. 

Pour simplifier on suppose $G$ quasi-déployé. 
On conjecture qu'il existe une factorisation  (dépendant d'une normalisation de  Whittaker globale)
\[\mathfrak A_\sigma=\bigotimes' \mathfrak A_{\sigma_v},\]
où  $\mathfrak A_{\sigma_v}$ est un espace vectoriel de dimension finie muni d'actions commutant entre elles  de   $\mathbb T_{K_v}^{\Qlbar}$ et de  $S_{\sigma_v}$, dépendant seulement de  $\sigma_v$. 
On conjecture que l'action de $S_\sigma$ sur  $\mathfrak A_\sigma$ 
est la restriction par 
 $S_\sigma\to \prod_v S_{\sigma_v}$ de l'action de $ \prod_v S_{\sigma_v}$ sur 
 $\bigotimes' \mathfrak A_{\sigma_v}$. Alors  $\mathfrak A_{\sigma_v}=\oplus \pi_v^{K_v}\otimes \rho_{\pi_v}$. En notant la fonction trace  $\on{tr}\rho_{\pi_v}$ sur  $S_{\sigma_v}$ par $\langle \pi_v, \cdot\rangle$, on aurait alors la factorisation 
 \[\langle \pi,x\rangle=\prod_v \langle \pi_v, x\rangle.\]
Si $v$ est une place non ramifiée,  on conjecture que  $\mathfrak A_{\sigma_{v}}$ est de dimension $1$, et que  $\mathbb T_{K_v}^{\Qlbar}$ agit par un caractère déterminé par  $\sigma_v$, et que  $S_{\sigma_v}$ agit par un caractère (dépendant de la normalisation de  Whittaker). En particulier, si $K_N=G(\mathbb O)$ est partout  hyperspécial, on s'attend à ce que 
$\mathfrak A_\sigma$ soit de dimension $1$, sur lequel l'algèbre de  Hecke agit par le caractère associé à  $\sigma$, et  $\mathfrak S_\sigma$ agit trivialement.  
A priori $\mathfrak S_\sigma$ pourrait agir par un caractère, 
mais on s'attend à ce qu'on puisse attacher à  $\sigma$ une forme automorphe, et donc ce caractère devrait être trivial. Il en résulte que l'on devrait avoir 
\[(H^{\cusp}_{I,W})_\sigma= \pi^{G(\mathbb O)}\otimes W_{\sigma^I}^{\mathfrak S_\sigma}.\]

\section{Non déformabilité des paramètres elliptiques}\label{non-deform}   
            
            Le reste de l'article est consacré au rappel de la construction des foncteurs $H$ et $H^{\cusp}$  et à la preuve de la proposition \ref{prop-ppale}. On commence par rappeler un lemme  connu. 
            
            Soit $\sigma$ un paramètre de Langlands vérifiant les conditions (C1), (C2), (C3), (C4), (C5) ci-dessus. 
            On rappelle que l'on note $S_{\sigma}$ le centralisateur de $\sigma$ dans $\wh G$ et  que 
  $\sigma$ est dit elliptique si   
 $\mathfrak S_{\sigma}:=S_{\sigma}/(Z_{\wh  G})^{ \on{Gal}(\wt F/F)}$ est fini.

 Il est évident que si $\sigma$ n'est pas elliptique il se déforme parmi les paramètres vérifiant (C1), (C2), (C3), (C4), (C5). En effet si $\sigma$ n'est pas elliptique $T:=\Ker( S_{\sigma}\to 
 (\wh G^{ab})^{ \on{Gal}(\wt F/F)})$ n'est pas fini (car il est isogène à $\mathfrak S_{\sigma}$). Or pour tout $g\in T(\Qlbar)$, $\sigma_{g}
$ défini par $\sigma_{g}(\gamma)=\sigma(\gamma) g^{\deg(\gamma)}$ vérifie 
 (C1), (C2), (C3), (C4), (C5) (et on peut même prendre  $g\in T(A)$ où $A$ est une  $\Qlbar$-algèbre arbitraire). 
 
 Ce qui nous intéresse est la réciproque. Elle est moins évidente, mais bien connue:  si $\sigma$  est elliptique, alors il ne peut pas être déformé parmi les morphismes vérifiant des conditions similaires à  (C1), (C2), (C3), (C4), (C5).   
 Plus précisément on a le lemme suivant.  
   
  \begin{lem} \label{lem-non-def} 
  On suppose $\sigma$   elliptique vérifiant    (C1), (C2), (C3), (C4), (C5).   
  Pour toute  $\Qlbar$-algèbre locale artinienne $\mc A$ d'idéal maximal $\mc I$, pour tout ouvert $U\subset U_0$, 
 tout morphisme continu $\pi_{1}(U, \overline \eta)\to {}^{L} G(\mc A)$ vérifiant les conditions (C1),   (C5)  
 (pour $ \mc A$ au lieu de 
$ \overline{{\mathbb Q}_\ell}$) 
 et dont la  réduction modulo $ \mc I$ est  $\sigma$,  est conjugué à  $\sigma$ par un élément de $\wh G(\mc A)$ égal à $1$ modulo $\mc I$. 
  \end{lem}
 \dem On rappelle la preuve d'après la proposition 5.12 de \cite{boeckle-harris...}
 (qui traite le cas déployé). Il suffit de traiter le cas où $\mc A=\Qlbar[\epsilon]/\epsilon^{2}$. 
Il suffit de considérer la   composition de ces morphismes avec  $\wh  G  \rtimes \on{Gal}(\wt F/F) \to  \wh  G^{ad} \rtimes \on{Gal}(\wt F/F) $ et nous sommes réduits au cas où  $\wh  G$ est semi-simple. 
Le système local  associé à la représentation adjointe   $\wh  {\mathfrak g}_{\sigma}$ est de poids  $0$ par la remarque 12.6 de \cite{coh}, donc  $H^1(U_{\overline{\mathbb F_{q}}}, \wh  {\mathfrak g}_{\sigma})$ a des poids  $\geq 1$ par Weil II, donc  
\begin{gather}\label{H0H1}H^0(\widehat {\mathbb Z},H^1(U_{\overline{\mathbb F_{q}}}, \wh  {\mathfrak g}_{\sigma}))=0.\end{gather} 
D'autre part  
$H^0(\widehat  {\mathbb Z} , H^0(U_{\overline{\mathbb F_{q}}}, \wh  {\mathfrak g}_{\sigma}))=H^0(U, \wh  {\mathfrak g}_{\sigma})$ est nul car 
$H^0(U, \wh  {\mathfrak g}_{\sigma})=\on{Lie}(S_{\sigma})$ et $S_{\sigma}$ est fini 
par l'ellipticité de  $\sigma$. Pour tout $\Qlbar$-espace vectoriel $M$ de dimension finie muni d'une action continue de $\widehat  {\mathbb Z}$, on a la suite exacte $$0\to H^{0}(\widehat  {\mathbb Z}, M)\to M\xrightarrow{1-u} M\to 
H^{1}(\widehat  {\mathbb Z}, M)\to 0,$$ où $u$ désigne l'action de $1\in \mathbb Z$,  donc $H^{0}(\widehat  {\mathbb Z}, M)$ et 
$H^{1}(\widehat  {\mathbb Z}, M)$ ont même dimension. 
Comme  on vient de voir que $H^0(\widehat  {\mathbb Z} , H^0(U_{\overline{\mathbb F_{q}}}, \wh  {\mathfrak g}_{\sigma}))=0$ on a donc 
\begin{gather}\label{H1H0}H^1(\widehat  {\mathbb Z} , H^0(U_{\overline{\mathbb F_{q}}}, \wh  {\mathfrak g}_{\sigma}))=0.\end{gather}  On déduit de \eqref{H0H1},  \eqref{H1H0} et de la suite spectrale de Leray que $H^1(U, \wh  {\mathfrak g}_{\sigma})=0$. \cqfd

 \section{Cohomologie des champs de chtoucas}\label{rappels-coho}
 
 On commence par rappeler des résultats de \cite{coh}.   On  définit,    suivant 12.3.2 de \cite{coh}, 
pour $I$ un ensemble fini, $k\in \N$, $(I_{1}, ..., I_{k})$ une  partition de $I$, 
et $W$ une représentation de $({}^{L} G)^{I}$, 
un champ   de Deligne-Mumford 
$$\Cht_{N,I,W}^{(I_{1},...,I_{k})}$$  sur $(X\sm  \wh N)^{I}$, réunion d'ouverts
$\Cht_{N,I,W}^{(I_{1},...,I_{k}), \leq \mu}$ tels que 
$\Cht_{N,I,W}^{(I_{1},...,I_{k}), \leq \mu}/\Xi$ soit de type fini. 
 De plus on  définit   $\mc F_{N,I,W, \Xi,E}^{(I_{1},...,I_{k})}$ comme le  faisceau pervers (à un décalage près) sur $\Cht_{N,I,W}^{(I_{1},...,I_{k})}/\Xi$ égal à l'image inverse d'un faisceau de Satake $\mc S_{I,W,E}^{(I_{1},...,I_{k})}$ par un morphisme lisse 
$$\epsilon_{N,(I),W,\underline{n}}^{(I_{1},...,I_{k}), \Xi}: 
\Cht_{N,I,W} ^{(I_{1},...,I_{k})}/\Xi
\to \mr{Gr}_{I,W}^{(I_{1},...,I_{k})}/G^{\mr{ad}}_{\sum _{i\in I}n_{i}x_{i}}$$
où les  entiers $n_{i}$ sont assez grands. 
Enfin on pose 
 \begin{gather} \label{defi-H-alpha-Cht}
 \mc H _{ N,I,W}^{\leq\mu,E}=  R
 \big(\mf p_{N,I}^{(I_{1},...,I_{k}),\leq\mu}\big)_{!}\Big(\restr{\mc F_{N,I,W,\Xi,E}^{(I_{1},...,I_{k})}} {\Cht_{N,I,W}^{(I_{1},...,I_{k}),\leq\mu}/\Xi}\Big)  
  \end{gather}
   qui appartient à $D^{b}_{c}((X\sm  \wh N)^{I}, E)$ 
    et ne dépend pas du choix de la   partition $(I_{1},...,I_{k})$.  

Pour tout point géométrique   $\ov x$  dans $(X\sm \wh N)^{I}$, 
on   définit le $E$-espace vectoriel gradué des  éléments  Hecke-finis   $$\Big(\varinjlim_{\mu }\restr{\mc H _{ N,I,W}^{\leq\mu,E}}{\ov x} \Big)^{\mr{Hf}}\subset \varinjlim_{\mu }\restr{\mc H _{ N,I,W}^{\leq\mu,E}}{\ov x} . $$ 
D'après \cite{coh}    \begin{gather}\label{Cccusp-Hf-non-deploye}\Big(\varinjlim_{\mu }\restr{\mc H _{ N,\{0\},\mbf 1}^{\leq\mu,E}}{\ov\eta}\Big)^{\mr{Hf}}=
       \Big(\varinjlim_{\mu }\restr{\mc H _{ N,\emptyset,\mbf 1}^{\leq\mu,E}}{\Fqbar}\Big)^{\mr{Hf}}=C_{c}^{\mr{cusp}}(\Bun_{G,N}(\Fq)/\Xi,E). \end{gather}
  D'après \cite{coh},   pour toute     flèche de spécialisation  $\on{\mf{sp}}$ de $\ov{\eta^{I}}$ vers $\Delta(\ov \eta)$,    
     $$\on{\mf{sp}}^{*}: \Big( \varinjlim _{\mu}\restr{\mc H _{N, I, W}^{\leq\mu,E}}{\Delta(\ov{\eta})} \Big)^{\mr{Hf}}\to 
 \Big( \varinjlim _{\mu}\restr{\mc H _{N, I, W}^{\leq\mu,E}}{\ov{\eta^{I}}}\Big)^{\mr{Hf}} $$
 est un isomorphisme  et  cet espace est muni 
      d'une action de 
       $\pi_{1}(\eta,\ov\eta)^{I}=\on{Gal}(\ov F/F)^{I}$, l'action sur le membre de gauche ne dépendant pas du choix de  $\ov{\eta^{I}}$ et de $\on{\mf{sp}}$. 
       Dans \cite{cong}, Cong Xue a montré que ces espaces sont de dimension finie {\it quand $G$ est déployé}. 
Une généralisation non écrite de ce théorème de Cong Xue affirme que ces espaces sont de dimension finie dans le cas général. {\it Nous admettons ce résultat}. Les opérateurs d'excursion agissent sur ces espaces   par  la construction suggérée dans \cite[12.3.4]{coh} et  détaillée dans \cite{cong-finite} et tout
$\sigma$ apparaissant dans la décomposition spectrale associée vérifie les conditions (C1), ..., (C5) ci-dessus (la condition (C5) a lieu grâce au quotient par $\Xi$). 

Dans le cas déployé,  Cong Xue a montré dans \cite{cong-finite} que pour toute     flèche de spécialisation  $\on{\mf{sp}}$ de $\ov{\eta^{I}}$ vers $\Delta(\ov \eta)$,    
     $$\on{\mf{sp}}^{*}:   \varinjlim _{\mu}\restr{\mc H _{N, I, W}^{\leq\mu,E}}{\Delta(\ov{\eta})}  \to 
 \varinjlim _{\mu}\restr{\mc H _{N, I, W}^{\leq\mu,E}}{\ov{\eta^{I}}} $$
 est un isomorphisme  et  que cet espace est muni 
      d'une action de 
       $\on{Weil}(\ov F/F)^{I}$, l'action sur le membre de gauche ne dépendant pas du choix de  $\ov{\eta^{I}}$ et de $\on{\mf{sp}}$. On s'attend à ce que ces résultats restent vrais dans le cas non déployé. 

On va maintenant introduire une notation nouvelle pour désigner certains sous-espaces  de  ces espaces de cohomologie. 

On considère $I$ de la forme $\{0\}\cup J$, où la réunion est disjointe. 
On note $(J_{1},...,J_{k})$ une partition de $J$. 
Soit $W$ une représentation de dimension finie de 
$\wh G \times ({}^{L} G)^{J}$. Il est clair que $W$ peut être incluse
dans une représentation de dimension finie $W'$ de 
$\wh G \times ({}^{L} G)^{J}$ qui s'étend à $({}^{L} G)^{\{0\}\cup J}$. 
Une variante du théoreme 12.16 de \cite{coh}  associe à $W$ un faisceau pervers $G^{\mr{ad}}_{\sum \infty x_{i}}$-équivariant
$\mc S_{\{0\}\cup J ,W,E}^{(\{0\}, J_{1},...,J_{k})}$
 sur $\restr{\mr{Gr}_{\{0\}\cup J }^{(\{0\},J_{1},...,J_{k})}}{\ov\eta\times U_0^{I}}$. 
En fait on pourrait remplacer $\ov\eta$ par le spectre du corps réflex
(de fa\c con tout à fait analogue à ce qu'on fait pour les variétés de Shimura, où $J$ est vide), 
mais on n'en a pas besoin ici. On définit alors  le champ 
$\Cht_{N,\{0\}\cup J,W}^{(\{0\},J_{1},...,J_{k}) }$
  sur $\ov\eta\times  (X\sm  \wh N)^{J}$ et sur $\Cht_{N,\{0\}\cup J,W}^{(\{0\},J_{1},...,J_{k}) }/\Xi$  on définit  un faisceau 
    $\mc F_{N,\{0\}\cup J,W, \Xi,E}^{(\{0\},J_{1},...,J_{k})}$   égal à l'image inverse de $\mc S_{\{0\}\cup J ,W,E}^{(\{0\}, J_{1},...,J_{k})}$ par le   morphisme lisse 
$$ \Cht_{N,\{0\}\cup J,W}^{(\{0\},J_{1},...,J_{k}) } /\Xi
\to \restr{\mr{Gr}_{\{0\}\cup J }^{(\{0\},J_{1},...,J_{k})}}{\ov\eta\times U_0^{J}}/G^{\mr{ad}}_{\sum _{i\in I}n_{i}x_{i}}.$$

On pose 
 \begin{gather} \nonumber 
 \mc H _{ N,\{0\}\cup J,W}^{\leq\mu,E}=  R
 \big(\mf p_{N,\{0\}\cup J}^{(\{0\},J_{1},...,J_{k}),\leq\mu}\big)_{!}\Big(\restr{\mc F_{N,\{0\}\cup J,W,\Xi,E}^{(\{0\},J_{1},...,J_{k})}} {\Cht_{N,\{0\}\cup J,W}^{(\{0\},J_{1},...,J_{k}),\leq\mu}/\Xi}\Big)  
  \end{gather}
   qui appartient à $D^{b}_{c}(\ov\eta\times (X\sm  \wh N)^{J}, E)$ 
    et ne dépend pas du choix de la   partition $(J_{1},...,J_{k})$.  
    
Alors    
     $$H^{\cusp}_{\{0\}\cup J,W} :=  \Big( \varinjlim _{\mu}\restr{\mc H _{N, \{0\}\cup J, W}^{\leq\mu,E}}{\Delta(\ov{\eta})} \Big)^{\mr{Hf}} $$ est muni 
      d'une action de 
       $\pi_{1}(\eta,\ov\eta)^{J}=\on{Gal}(\ov F/F)^{J}$.
       En effet c'est un sous-espace de $  \Big( \varinjlim _{\mu}\restr{\mc H _{N, \{0\}\cup J, W'}^{\leq\mu,E}}{\Delta(\ov{\eta})} \Big)^{\mr{Hf}} $
       (qui est muni d'une action de $\on{Gal}(\ov F/F)^{\{0\}\cup J}$). 
       En fait, en notant $\check F$ le corps réflex, on pourrait munir 
        $ H^{\cusp}_{\{0\}\cup J,W}$ d'une action de $\on{Gal}(\ov F/\check F) \times \on{Gal}(\ov F/F)^{J}$, mais on n'en a pas besoin ici. 
        
        On définit  de plus
         $$H_{\{0\}\cup J,W} :=   \varinjlim _{\mu}\restr{\mc H _{N, \{0\}\cup J, W}^{\leq\mu,E}}{\Delta(\ov{\eta})} $$ qui, dans le cas déployé d'après \cite{cong-finite} et conjecturalement en général, est muni 
      d'une action de 
       $\on{Weil}(\ov F/F)^{J}$.

       \section{Une construction proposée par Drinfeld}
       
        On explique maintenant la construction proposée par Drinfeld, en étendant la remarque 8.5 de \cite{ICM-text} au cas non déployé. 
        
        Soit $\on{Reg}$ la représentation régulière gauche de  $\widehat G$ à  coefficients dans  $E$ (considérée 
        comme une limite inductive de représentations de dimension finie). 
                On a  \begin{gather}\label{dec-H0reg-def}H^{\cusp}_{\{0\},\on{Reg}}=\bigoplus H^{\cusp}_{\{0\},V} \otimes V^{*}\end{gather} où la somme directe est prise sur les représentations irreductibles de  $\wh  G$ à coefficients dans $E$.  On remarque que  $H^{\cusp}_{\{0\},V}$ est nul lorsque le centre $(Z_{\wh  G})^{ \on{Gal}(\wt F/F)}$ agit non trivialement sur $V$, si bien que  $(Z_{\wh  G})^{ \on{Gal}(\wt F/F)}$ agit trivialement sur  $H^{\cusp}_{\{0\},\on{Reg}}$.

On va munir le  $E$-espace vectoriel gradué $H^{\cusp}_{\{0\},\on{Reg}}$  
 \begin{itemize}
 \item  [a) ] 
d'une structure de module sur 
l'algèbre des fonctions sur 
   ``l'espace affine $\mathcal S$ des  morphismes $\sigma:\on{Gal}(\overline F/F)\to 
{}^{L} G$ à coefficients dans les  $E$-algèbres vérifiant des conditions analogues à (C3), (C4), (C5)'', 
 \item [b)]  d'une  action algébrique de  $\widehat G$ (venant de l'action à droite de  
 $\widehat G$ sur $\on{Reg}$) qui est compatible avec la conjugaison par   $\widehat G$ 
sur $\mathcal S$. 
 \end{itemize}
 L'espace $\mathcal S$ n'est pas défini rigoureusement et la définition rigoureuse de la structure a) est la suivante. 
 Pour toute représentation $E$-linéaire de dimension finie  $V$ de  $ {}^{L} G$, d'espace vectoriel sous-jacent $\underline V$, 
  $H^{\cusp}_{\{0\},\on{Reg}}\otimes \underline V$ est muni d'une  action de $\on{Gal}(\overline F/F)$, qui en fait une limite inductive de représentations continues de dimension finie de  $\on{Gal}(\overline  F/F)$, de la fa\c con suivante. 
On possède l'isomorphisme   $\widehat  G$-équivariant  
   \begin{align*} \theta:  \on{Reg} \otimes \underline V  & \simeq  \on{Reg}\otimes V  \\
     f\otimes x & \mapsto    [g\mapsto f(g)  g.x]
  \end{align*} 
   où   $\widehat  G$ agit diagonalement sur le membre de droite, et où pour donner un sens à la formule on identifie ce dernier 
   à l'espace vectoriel des fonctions algébriques $\widehat G\to V$. 
 On en déduit un  isomorphisme  
    $$  H^{\cusp}_{\{0\},\on{Reg}} \otimes  \underline   V = H^{\cusp}_{\{0\},\on{Reg}  \otimes \underline V }
    \isor{\theta} H^{\cusp}_{\{0\},\on{Reg}  \otimes V}
    \simeq H^{\cusp}_{\{0\}\cup\{1\},\on{Reg} \boxtimes V}$$
   où la première égalité est tautologique (puisque  $ \underline V$ est simplement un espace vectoriel) et le dernier  isomorphisme est l'inverse de l'isomorphisme de coalescence.  
     Alors   l'action de  $ \on{Gal}(\overline F/F)$ sur le membre de gauche   
est définie comme étant l'action  de  $ \on{Gal}(\overline F/F)$ sur le membre de droite correspondant à la patte  $1$.  
 Si   $V_{1}$ et  $V_{2}$ sont deux  représentations de  ${}^{L} G$, 
les deux  actions de  $ \on{Gal}(\overline F/F)$ sur   $H^{\cusp}_{\{0\},\on{Reg}}\otimes  \underline  {V_{1}}\otimes  \underline  {V_{2}}$ associées aux actions de $\widehat G$ sur $V_{1}$ et $V_{2}$  commutent entre elles  et l'action diagonale de  $ \on{Gal}(\overline F/F)$ est égale à  l'action  
 associée à l'action diagonale  de  $\widehat G$ sur 
   $V_{1}\otimes V_{2}$.  
 Cela donne la  structure   a) car si $V$ est comme ci-dessus, $x\in V$, $\xi\in V^{*}$,  $f$ est la  fonction sur  ${}^{L} G$ définie comme le  coefficient de matrice $f(g)=\langle \xi, g.x \rangle$, et $\gamma\in  \on{Gal}(\overline F/F)$ on pose que    $F_{f,\gamma}: \sigma\mapsto f(\sigma(\gamma))$, considérée comme une  ``fonction sur $\mathcal S$'',   agit sur  $H^{\cusp}_{\{0\},\on{Reg}}$ par la composée    
 \begin{gather} \label{compo-Ffgamma}H^{\cusp}_{\{0\},\on{Reg}} \xrightarrow{\on{Id}\otimes x} H^{\cusp}_{\{0\},\on{Reg}} \otimes \underline V \xrightarrow{\gamma} H^{\cusp}_{\{0\},\on{Reg}} \otimes \underline V 
 \xrightarrow{\on{Id}\otimes  \xi} H^{\cusp}_{\{0\},\on{Reg}}. 
 \end{gather}  
 Toute fonction $f$ sur ${}^{L} G$ peut être écrite comme un  coefficient de matrice, 
 et les  fonctions  $F_{f,\gamma}$ quand  $f$ et $\gamma$ varient 
 sont supposées   ``engendrer topologiquement toutes les fonctions sur  $\mathcal S$''. La propriété ci-dessus avec  $V_{1}$ et $V_{2}$ implique les  relations entre les   $F_{f,\gamma}$, à savoir que 
 \begin{gather}\label{formule-coproduit}
 F_{f,\gamma_{1}\gamma_{2}}=\sum_{\alpha }F_{f_{1}^{\alpha},\gamma_{1}}
 F_{f_{2}^{\alpha},\gamma_{2}}\end{gather} si  l'image de $f$ par le coproduit est  
 $\sum_{\alpha} f_{1}^{\alpha}\otimes f_{2}^{\alpha}$. 
Dans  \cite{zhu-notes},  le second auteur donne une  construction équivalente de la structure a). 
Les structures a) et  b) sont  compatibles au sens suivant: la  conjugaison $g F_{f,\gamma} g^{-1}$
de l'action de $F_{f,\gamma}$ sur $H^{\cusp}_{\{0\},\on{Reg}}$ par l'action algébrique de  $g\in \widehat G$ est égale à l'action de $F_{f^{g},\gamma} $ où  
$f^{g}(h)=f(g^{-1}hg)$.

\begin{rem}
On peut définir des ``fonctions sur $\mathcal S$'' plus générales 
$F_{f,(\gamma_{1}, ..., \gamma_{n})}:\sigma\mapsto f(\sigma(\gamma_{1}), ..., \sigma(\gamma_{n}))$ avec 
$f$ une fonction sur $({}^{L} G)^{n}$ et $\gamma_{1}, ..., \gamma_{n}\in  \on{Gal}(\overline F/F)$. On peut les faire agir  sur  $H^{\cusp}_{\{0\},\on{Reg}}$,   en généralisant l'action des $F_{f,\gamma}$, par une construction généralisant  \eqref{compo-Ffgamma}.  
Cependant toute fonction sur $({}^{L} G)^{n}$ est une somme de fonctions de la forme $(g_{1}, ..., g_{n})\mapsto f_{1}(g_{1})... f_{n}(g_{n})$ et pour $f$ de cette forme $F_{f,(\gamma_{1}, ..., \gamma_{n})}$ est égal à 
$F_{f_{1},\gamma_{1}}... F_{f_{n},\gamma_{n}}$. Donc les  $F_{f,(\gamma_{1}, ..., \gamma_{n})}$ appartiennent en fait à la sous-algèbre engendrée par les $F_{f,\gamma}$ dans l'algèbre des endomorphismes de  $H^{\cusp}_{\{0\},\on{Reg}}$. Les  $F_{f,(\gamma_{1}, ..., \gamma_{n})}$ permettent cependant de préciser le lien avec les opérateurs d'excursion: 
 si $f$ est  une fonction sur $({}^{L} G)^{n}$ invariante par conjugaison par $\wh G$,  $F_{f,(\gamma_{1}, ..., \gamma_{n})}$ agit sur $H^{\cusp}_{\{0\},\on{Reg}}$ en respectant la décomposition~\eqref{dec-H0reg-def} et agit sur chaque facteur  
 $H^{\cusp}_{\{0\},V} \otimes V^{*}$ par $S\otimes \Id_{V^{*}}$ où $S$ est un opérateur d'excursion,  associé à $f$ et $(\gamma_{1}, ..., \gamma_{n})$, qui agit  sur $H^{\cusp}_{\{0\},V}$ comme dans la section 3.7 de \cite{cong-finite}. 
\end{rem}

  \begin{lem}
 Soit $c\in H^{\cusp}_{\{0\},\on{Reg}}$ et  $f$ une fonction algébrique sur  ${}^{L}G$. L'espace engendré par  $\{F_{f,\gamma}c\mid \gamma\in\Gamma\}$ est de dimension finie et l'application  $\gamma\mapsto F_{f,\gamma}c$ est continue. De plus, pour presque toute place $v$ de $X$, $F_{f,\gamma}c=f(1)c$ pour  $\gamma$ dans le sous-groupe d'inertie  $I_v$ en cette place.
 \end{lem}
 \begin{proof}
 Sans perdre de généralité,  on peut supposer que  $f(g)=\langle \xi, gx\rangle$ avec $x\in V$ et $\xi \in V^*$.
 On écrit  $\on{Reg}$ comme une réunion croissante de représentations $W_i$ de $\hat{G}$. Alors $(\on{Id}\otimes x)(c)\in H^{\cusp}_{\{0\},W_i}\otimes \underline V$ pour une certaine représentation de dimension finie   $W_i$. Le lemme en résulte immédiatement. 
 \end{proof}

 Les structures a) et b) fournissent un  ``$\mathcal O$-module sur le champ $\mathcal S/\widehat G$ des paramètres de Langlands globaux''  
 (tel que l'espace vectoriel de ses  ``sections globales sur  $\mathcal S$''  est  $H^{\cusp}_{\{0\},\on{Reg}} $). 
 Plus précisément l'action de $\widehat G$ sur $\mathcal S$ se factorise par 
 $\widehat G/(Z_{\wh  G})^{ \on{Gal}(\wt F/F)}$ et on a un ``$\mathcal O$-module sur le champ $\mathcal S/\big(\widehat G/(Z_{\wh  G})^{ \on{Gal}(\wt F/F)}\big)$  tel que l'espace vectoriel de ses  ``sections globales sur  $\mathcal S$''  est  $H^{\cusp}_{\{0\},\on{Reg}} $. 
 
Pour tout  morphisme 
 $\sigma: \on{Gal}(\overline F/F)\to {}^{L} G(\overline{ \mathbb Q_{\ell}})$, 
 on veut définir $\mathfrak A_{\sigma}$ comme la fibre de ce  $\mathcal O$-module au-dessus de  $\sigma$ (considéré comme un   ``$\overline{ \mathbb Q_{\ell}}$-point de $\mathcal S$ dont le groupe des  automorphismes dans le champ    $\mathcal S/\big(\widehat G/(Z_{\wh  G})^{ \on{Gal}(\wt F/F)}\big)$ est $\mathfrak S_{\sigma}=S_{\sigma}/(Z_{\wh  G})^{ \on{Gal}(\wt F/F)}$'').

 Rigoureusement on définit    $\mathfrak A_{\sigma}$ comme le plus grand quotient  de $H^{\cusp}_{\{0\},\on{Reg}}\otimes_{\mathbb Q_{\ell}}  \overline{ \mathbb Q_{\ell}}$ sur lequel, pour tout $f$ et $\gamma$,     $F_{f,\gamma}$ agit par   multiplication par le scalaire  $f(\sigma(\gamma))$. 
  On voit que    $\mathfrak S_{\sigma}=S_{\sigma}/(Z_{\wh  G})^{ \on{Gal}(\wt F/F)}$ agit sur $\mathfrak A_{\sigma}$. 
  
  Si l'heuristique (8.3) de \cite{ICM-text} est vraie c'est le même   $\mathfrak A_{\sigma}$ que dans l'heuristique. 
 
 \begin{rem}
 La construction précédente s'applique aussi au foncteur $H$ (au lieu de $H^{\cusp}$), dans le cas déployé (et conjecturalement en général), en rempla\c cant dans a) $ \on{Gal}(\overline F/F)$ par  $ \on{Weil}(\overline F/F)$.  \end{rem}
 
 \section{Etude du cas elliptique}
 
 On rappelle que tout ensemble fini $I$ et toute représentation $W$ de dimension finie de $({}^{L}G)^{I}$, on a admis que $H^{\cusp}_{I,W}$ est de dimension finie
(ce qui est une extension non encore écrite de  \cite{cong}, 
où ce résultat est montré lorsque $G$ est déployé). 
Pour tout paramètre de Langlands $\sigma$,  on note  alors  $(H^{\cusp}_{I,W})_{\sigma}$ l'espace propre  {\it généralisé}   
dans  $H^{\cusp}_{I,W}\otimes_{E}\overline{{\mathbb Q}_\ell} 
$
 associé  à  $\sigma$ pour l'action des opérateurs  d'excursion
 et le système de valeurs propres de ces opérateurs correspondant à $\sigma$. 

La proposition suivante implique la proposition \ref{prop-ppale}. 

  \begin{prop}
  Pour tout $\sigma$ elliptique on a l'égalité 
 $$(H^{\cusp}_{I,W})_{\sigma}= \big( \mathfrak A_{\sigma}\otimes_{\overline{{\mathbb Q}_\ell}} W_{\sigma^{I}}\big)^{S_{\sigma}} ,$$ équivariante par l'action de $\on{Gal}(\overline F/F)^{I}$, 
 fonctorielle en $W$, et compatible avec les isomorphismes de coalescence. 
Autrement dit l'heuristique  (8.3) de  \cite{ICM-text} est vraie au-dessus de $\sigma$.  
  \end{prop}

    \noindent {\bf Démonstration.} On a  $H^{\cusp}_{\{0\},\on{Reg}}=\bigoplus H^{\cusp}_{\{0\},V} \otimes V^{*}$ où la somme directe est prise sur les représentations irreductibles de  $\wh  G$. On définit 
    \begin{gather}\label{Hregsigma}(H^{\cusp}_{\{0\},\on{Reg}})_{\sigma}=\bigoplus (H^{\cusp}_{\{0\},V})_{\sigma} \otimes V^{*}. \end{gather}
 On  a déjà vu que  $(Z_{\wh  G})^{ \on{Gal}(\wt F/F)}$ agit trivialement sur  $H^{\cusp}_{\{0\},\on{Reg}}$ et donc sur  $(H^{\cusp}_{\{0\},\on{Reg}})_{\sigma}$. 
 
 Les $F_{f,\gamma}$ agissent sur  $(H^{\cusp}_{\{0\},\on{Reg}})_{\sigma}$
 (parce qu'ils commutent avec les opérateurs  d'excursion). 
 
On choisit  $n$ et  $(\gamma_1, ..., \gamma_n)$ tels que 

(H)
$\sigma(\gamma_1)$, ..., $\sigma(\gamma_n)$ 
engendrent un sous-groupe Zariski dense de l'image of $\sigma$.

Cela est possible car l'image de $\sigma$ est topologiquement de type fini
(même argument que dans la preuve du lemme 3.2.7 de \cite{cong-finite}). 

Par l'hypothèse d'ellipticité de  $\sigma$, 
et grâce à (H), le $n$-uplet  $(\sigma(\gamma_1)$, ..., $\sigma(\gamma_n))$  est semi-simple et les résultats de Richardson \cite{richardson} rappelés dans le lemme 11.9 de \cite{coh}  impliquent que 
la  $\wh  G$-orbite par  conjugaison  
de $(\sigma(\gamma_1)$, ..., $\sigma(\gamma_n))$ 
dans $({}^L   G)^n$
est fermée, et égale à  $\wh  G/S_{\sigma}$. De plus cette orbite est schématiquement égale à l'image inverse de l'image de  $(\sigma(\gamma_1), ..., \sigma(\gamma_n))$  par le morphisme de $({}^L   G)^n$ dans le quotient grossier de $({}^L   G)^n$ par conjugaison par $\wh G$. 

Grâce à l'action des   $F_{f,\gamma_{i}}$ pour  $i=1,...,n$ et $f$ fonction sur ${}^L G$, 
$(H^{\cusp}_{\{0\},\on{Reg}})_{\sigma}$ est un  module sur  $\mathcal O(({}^L G)^n)$.
Toute fonction $\wh  G$-invariante sur $({}^L G)^n$ s'annulant sur  l'orbite $\on{Orb}=\wh  G (\sigma(\gamma_1), ..., \sigma(\gamma_n))$ agit de fa\c con nilpotente sur chaque élément de  
$(H^{\cusp}_{\{0\},\on{Reg}})_{\sigma}$ 
(car c'est un opérateur d'excursion s'annulant sur la classe de conjugaison de $\sigma$) 
et donc 
$(H^{\cusp}_{\{0\},\on{Reg}})_{\sigma}$ est un module sur  $\mathcal O(\widehat{\on{Orb}})$
où  $\widehat{\on{Orb}}$ est la complétion formelle  de $({}^L G)^n$ le long de l'orbite  $\on{Orb}$. 

Le théorème de la tranche étale de  Luna implique qu'il existe 
une variété affine localement fermée   $Z$ dans 
$({}^L  G)^n$, contenant $(\sigma(\gamma_1)$, ..., $\sigma(\gamma_n))$ et  stable par conjugaison par $S_{\sigma}$, telle que  
$Z\times_{S_{\sigma}}\wh  G\to ({}^L  G)^n$ est (fortement) étale. 
Soit $\widehat Z$  la complétion formelle de  $Z$ en $(\sigma(\gamma_1), ..., \sigma(\gamma_n))$. On a alors 
\begin{gather}\label{Orb-Z}\widehat{\on{Orb}}=\widehat Z\widehat{\times}_{S_{\sigma}}\wh  G.\end{gather} On définit  
\begin{gather}\label{defi-sigma-Z}
(H^{\cusp}_{\{0\},\on{Reg}})_{\sigma, Z}=(H^{\cusp}_{\{0\},\on{Reg}})_{\sigma}\otimes_{\mathcal O(\widehat{\on{Orb}})} \mathcal O(\widehat Z),\end{gather}
 autrement dit on quotiente par l'idéal de $\widehat Z$ dans $\widehat{\on{Orb}}$.

Il résulte de \eqref{Hregsigma} que pour toute représentation de dimension finie  $V$ de $\wh  G$
on a   
 $$(H^{\cusp}_{\{0\},V})_{\sigma}= \big( (H^{\cusp}_{\{0\},\on{Reg}})_{\sigma} \otimes_{\overline{{\mathbb Q}_\ell}} V\big)^{\wh  G} $$
De  \eqref{Orb-Z} on déduit que 
 \begin{gather}\label{dec-V-sigma}(H^{\cusp}_{\{0\},V})_{\sigma}= \big( (H^{\cusp}_{\{0\},\on{Reg}})_{\sigma, Z} \otimes_{\overline{{\mathbb Q}_\ell}} V\big)^{S_{\sigma}}. \end{gather}

 On rappelle qu'on a admis que  tous les  $H^{\cusp}_{\{0\},V}$ sont de dimension finie
 (résultat montré dans \cite{cong} lorsque $G$ est déployé). 
Comme $S_{\sigma}$ est fini
modulo   $(Z_{\wh  G})^{ \on{Gal}(\wt F/F)}$ (qui agit trivialement sur  $(H^{\cusp}_{\{0\},\on{Reg}})_{\sigma, Z}$) il existe une représentation $V$ de dimension finie de $\wh G/(Z_{\wh  G})^{ \on{Gal}(\wt F/F)}$  telle que la restriction de $V$ au groupe fini  $\mathfrak S_{\sigma}=S_{\sigma}/(Z_{\wh  G})^{ \on{Gal}(\wt F/F)}$ contient toutes les représentations irréductibles de ce groupe. Choisissant une telle représentation $V$ 
 on déduit  de la finitude de la dimension de  $H^{\cusp}_{\{0\},V}$  et  de \eqref{dec-V-sigma} que  $(H^{\cusp}_{\{0\},\on{Reg}})_{\sigma, Z}$ est de   dimension finie. 
 Il  en résulte aussi  que le seul paramètre de Langlands  pouvant apparaître dans la décomposition spectrale de  $(H^{\cusp}_{\{0\},\on{Reg}})_{\sigma, Z}$ par l'action des opérateurs d'excursion est la classe de conjugaison de $\sigma$.

On note  $\mathcal C$ la sous-algèbre commutative  de $\on{End}((H^{\cusp}_{\{0\},\on{Reg}})_{\sigma, Z})$  engendrée par l'action de tous les $F_{f,\gamma}$ quand $f$ et $\gamma$ varient. Grâce à la  structure a) on a  (quitte à restreindre l'ouvert $U$  de $X$) un morphisme  
$$\beta:\pi_{1}(U, \overline \eta)\to {}^L  G(\mathcal C) $$
défini par la propriété que  pour toute  
 fonction algébrique $f$  sur   ${}^L  G$ et pour tout $\gamma\in \on{Gal}(\ov F/F)$, on a 
  $f(\beta(\gamma))=F_{f,\gamma}$ dans $  \mathcal C$. On déduit de 
  \eqref{formule-coproduit} que $\beta$ est un morphisme de groupes. 
 La composition de $\beta$  avec  \eqref{compo-ab} est d'ordre fini fixé 
(à cause du quotient par $\Xi$ dans la  définition des $H^{\cusp}_{I,W}$). 
Le morphisme $\beta$  est continu parce que pour toute  
 fonction algébrique $f$  sur   ${}^L  G$, le morphisme 
  $\gamma\mapsto f(\beta(\gamma))$ 
est une application continue de  $\pi_{1}(U, \overline \eta)$  vers  $\mathcal C$.

On a un morphisme $ \mathcal O(( {}^{L}G)^{n})\to \mathcal C$ qui envoie 
$f_{1}\otimes ... \otimes f_{n}$ sur $F_{f_{1},\gamma_{1}}... F_{f_{n},\gamma_{n}}$. Il résulte de \eqref{defi-sigma-Z} que ce morphisme se factorise à travers 
le quotient $ \mathcal O(( {}^{L}G)^{n}) \to \mathcal O(\widehat Z)$.  
On a donc un morphisme $ \mathcal O(\widehat Z)\to \mathcal C$.  
Comme $\mathcal C$ est de dimension finie, pour $k$ entier assez grand ce morphisme se factorise par $ \mathcal O(Z_{k})$ où $Z_{k}$ désigne le voisinage épaissi d'ordre $k$ de $(\sigma(\gamma_1), ..., \sigma(\gamma_n))$ 
dans $Z$. On a donc un morphisme $\mathcal O(Z_{k})\to \mathcal C$. 
On a de plus un diagramme commutatif
\begin{gather}\label{diag-com}
 \xymatrix{
\on{Spec} \mathcal C  \ar[rr]^{(\beta(\gamma_{1}), ..., \beta(\gamma_{n}))}   \ar[rd] 
&& ({}^{L}G)^{n} 
 \\
& \on{Spec} Z_{k}\ar[ur] }\end{gather}

Grâce à (H), le centralisateur de 
$(\sigma(\gamma_{1}), ..., \sigma(\gamma_{n}))$ est égal à celui de $\sigma$.   Il en résulte que pour tout $\sigma'$ tel que 
 \begin{itemize}
 \item $\sigma'$ 
  est conjugué à $\sigma$, 
 \item $(\sigma'(\gamma_1), ..., \sigma'(\gamma_n))$ est égal à  
  $(\sigma(\gamma_1), ..., \sigma(\gamma_n))$,  
  \end{itemize}
  alors 
 $\sigma'$ est égal à $\sigma$.
 
 Pour tout caractère $\chi: \mathcal C\to \Qlbar$, 
 $\chi\circ \beta$ vérifie ces deux conditions donc est égal à $\sigma$. 
En d'autres termes 
 le réduit du spectre de $\mathcal C$ est égal au seul point $\sigma$. 
 En particulier $\mathcal C$ est une $\Qlbar$-algèbre locale artinienne. 
 
D'après le lemme~\ref{lem-non-def},  $\beta$ est conjugué à  $\sigma$. 
En particulier $(\beta(\gamma_{1}), ..., \beta(\gamma_{n}))$ est conjugué à 
$(\sigma(\gamma_{1}), ..., \sigma(\gamma_{n}))$, et comme   $(\beta(\gamma_{1}), ..., \beta(\gamma_{n}))$ est au-dessus de la tranche  $Z_{k}$ il est égal à  $(\sigma(\gamma_{1}), ..., \sigma(\gamma_{n}))$, puisque $\on{Orb}\cap Z_{k}=\{(\sigma(\gamma_{1}), ..., \sigma(\gamma_{n}))\}$ schématiquement. Comme le centralisateur de 
$(\sigma(\gamma_{1}), ..., \sigma(\gamma_{n}))$ est égal à celui de $\sigma$ on a alors  $\beta=\sigma$ et   $\mathcal C=\Qlbar$. 

On en déduit   
$(H^{\cusp}_{\{0\},\on{Reg}})_{\sigma, Z} =\mathfrak A_{\sigma}$, 
ce qui conclut la preuve de la proposition \ref{prop-ppale}. En effet 
pour tout    $I$ et pour toute représentation de dimension finie $W$ de $({}^L  G)^{I}$, 
$$(H^{\cusp}_{I,W})_{\sigma}= \big( (H^{\cusp}_{\{0\},\on{Reg}})_{\sigma, Z} \otimes_{\overline{{\mathbb Q}_\ell}} W_{\sigma^{I}}\big)^{S_{\sigma}}. $$
Une autre formulation est que $\mathfrak A_{\sigma}$, muni de l'action de $S_{\sigma}$, définit un fibré vectoriel $\wh G$-équivariant sur $\wh G/S_{\sigma}$, orbite de $\sigma$ par conjugaison par $\wh G$ (vue comme schéma affine sur $\Qlbar$), dont les sections globales sont $(H^{\cusp}_{\{0\},\on{Reg}})_{\sigma} $.

\end{document}